\documentclass[a4paper,12pt]{amsart}

\def\NZQ{\mathbb}               % the font for N,Z,Q,R,C
\def\NN{{\NZQ N}}

\def\ZZ{{\NZQ Z}}

\def\CC{{\NZQ C}}
\def\F2{{\NZQ F}_2}

\def\opn#1#2{\def#1{\operatorname{#2}}} % to make operators
\opn\chara{char} \opn\length{\ell} \opn\pd{pd} \opn\rk{rk}
\opn\projdim{proj\,dim} \opn\injdim{inj\,dim} \opn\rank{rank}
\opn\depth{depth} \opn\codepth{codepth} \opn\grade{grade}
\opn\height{height} \opn\embdim{emb\,dim} \opn\codim{codim}

\opn\Tr{Tr} \opn\bigrank{big\,rank}
\opn\superheight{superheight}\opn\lcm{lcm}
\opn\trdeg{tr\,deg}%
\opn\reg{reg} \opn\lreg{lreg} \opn\skel{skel}
\opn\Gr{Gr}
\opn\ann{ann}
\opn\sign{sign}
\opn\del{del}

\opn\div{div} \opn\Div{Div} \opn\cl{cl} \opn\Cl{Cl}
\opn\Spec{Spec} \opn\Supp{Supp} \opn\supp{supp} \opn\Sing{Sing}
\opn\Ass{Ass}\opn\fdepth{fdepth}
\opn\Ann{Ann} \opn\Rad{Rad} \opn\Soc{Soc}
\opn\Sym{Sym} \opn\Ker{Ker} \opn\Coker{Coker} \opn\Im{Im}
\opn\Hom{Hom} \opn\Tor{Tor} \opn\Ext{Ext} \opn\End{End}
\opn\Aut{Aut} \opn\id{id} \opn\ini{in} \opn\tr{tr}

\opn\nat{nat}\opn\it{it}
\opn\pff{proof}%   \pf exists already
\opn\Pf{proof} \opn\GL{GL} \opn\SL{SL} \opn\mod{mod} \opn\ord{ord}
\opn\aff{aff} \opn\con{conv} \opn\relint{relint} \opn\st{st}
\opn\lk{lk} \opn\cn{cn} \opn\core{core} \opn\vol{vol}
\opn\link{link} \opn\star{star} \opn\skel{skel} \opn\indeg{indeg}
\opn\Ass{Ass} \opn\Min{Min} \opn\sdepth{sdepth} \opn\depth{depth}
\opn\gr{gr}

\def\pot#1#2{#1[\kern-0.28ex[#2]\kern-0.28ex]}

\opn\dirlim{\underrightarrow{\lim}}
\opn\inivlim{\underleftarrow{\lim}}
\def\Implies{\ifmmode\Longrightarrow \else
     \unskip${}\Longrightarrow{}$\ignorespaces\fi}
\def\implies{\ifmmode\Rightarrow \else
     \unskip${}\Rightarrow{}$\ignorespaces\fi}
\def\iff{\ifmmode\Longleftrightarrow \else
     \unskip${}\Longleftrightarrow{}$\ignorespaces\fi}

\let\:=\colon
\theoremstyle{plain}
\newtheorem{Theorem}{Theorem}[section]
 \newtheorem{Lemma}[Theorem]{Lemma}
 \newtheorem{Corollary}[Theorem]{Corollary}
 \newtheorem{Proposition}[Theorem]{Proposition}
 
 \newtheorem{Conjecture}[Theorem]{Conjecture}
 \newtheorem{Question}[Theorem]{Question}
 \theoremstyle{definition}
 \newtheorem{Definition}[Theorem]{Definition}
 \newtheorem{Remark}[Theorem]{Remark}
 
 \newtheorem{Example}[Theorem]{Example}

\let\epsilon\varepsilon
\let\kappa=\varkappa
\opn\dis{dis}
\def\pnt{{\raise0.5mm\hbox{\large\bf.}}}

\opn\Lex{Lex}
\usepackage{subcaption,tikz}

\usepackage{float}
\usepackage{tikz-cd}
\usepackage{mathtools}
\usepackage{amsfonts}
\usepackage{amssymb}
\newcommand{\HP}{\mathrm{HP}}

\newcommand{\PP}{\mathcal{P}}
\newcommand{\MM}{\mathcal{M}}

\renewcommand{\H}{\mathrm{H}}

\renewcommand{\phi}{\varphi}

\renewcommand{\CC}{\mathcal{C}}
\newcommand{\DD}{\mathcal{D}}
\title{Hilbert Series of simple thin polyominoes}
\author[Giancarlo Rinaldo]{Giancarlo Rinaldo}

\address[Giancarlo Rinaldo]{Department of Mathematics\\
University of Trento\\
via Sommarive, 14\\
38123 Povo (Trento), Italy}
\email{giancarlo.rinaldo@unitn.it}

\author[Francesco Romeo]{Francesco Romeo}

\address[Francesco Romeo]{Department of Mathematics\\
University of Trento\\
via Sommarive, 14\\
38123 Povo (Trento), Italy}
\email{francesco.romeo-3@unitn.it}

\begin{document}

\maketitle
 
 \begin{abstract}
 Let $\PP$ be a simple thin polyomino, namely a polyomino that has no holes and does not contain a square tetromino as a subpolyomino. In this paper, we determine the reduced Hilbert-Poincar\'{e} series $h(t)/(1-t)^d$ of $K[\PP]$ by proving that $h(t)$ is the rook polynomial of $\PP$. As an application, we characterize the Gorenstein simple thin polyominoes.
 \end{abstract}
\section{Introduction}
Polyominoes are two-dimensional objects obtained by joining edge by edge squares of same size, and they are studied from the point of view of combinatorics, e.g. in tiling problems of the plane (see \cite{Go}). Recently in \cite{Qu}, Qureshi introduced a binomial ideal induced by the geometry of a given polyomino $\PP$, called polyomino ideal, and the related algebra $K[\PP]$ (see Section \ref{sec:pre}). From that moment different authors studied algebraic properties related to this ideal (see \cite{HM,QSS,Sh,MRR1}). In particular in \cite{HM,QSS} the authors proved that if $\PP$ is simple, namely the polyomino has no holes, then $K[\PP]$ is a Cohen-Macaulay domain.

In this paper we compare two generating functions associated with polyominoes: the Hilbert series of $K[\PP]$ and the rook polynomial of $\PP$ (see \cite[Chapter 7]{Ri}).
The well-known ``rook problem'' is the problem of enumerating the number of ways of placing $k$ non-attacking rooks on a chessboard. In a similar way, let $\PP$ be a polyomino and let $r_k$ be the number of ways of arranging $k$ non-attacking rooks on the cells of $\PP$. The polynomial 
\[
r_{\PP}(t)=\sum_{k=0}^{r(\PP)} r_k t^k
\] 
is called the {\em rook polynomial} of $\PP$ and $r(\PP)$ is called the \emph{rook number} of $\PP$.

In a recent paper \cite{EHQR}, the authors proved that, for particular convex polyominoes $\PP$, the Castelnuovo-Mumford regularity of $K[\PP]$ is equal to $r(\PP)$. Starting from this result, we consider the Hilbert-Poincar\'e series of simple polyominoes as a nice object to grasp the above equality and other fundamental invariants by using elementary proofs.

%In this paper, we prove the above equality also for the class of simple thin polyominoes.
We say that a polyomino $\PP$  is \emph{thin} (see \cite{MRR2}) if $\PP$ does not contain the square tetromino (see Figure \ref{fig:square}) as a subpolyomino.
\begin{figure}[H]
\centering
\begin{tikzpicture}
\draw (0,0)--(2,0);
\draw (0,1)--(2,1);
\draw (0,2)--(2,2);

\draw (0,0)--(0,2);
\draw (1,0)--(1,2);
\draw (2,0)--(2,2);
\end{tikzpicture}\caption{The square tetromino}\label{fig:square}
\end{figure}
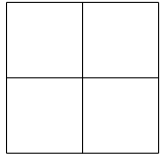
One of the main results of this paper is the following
\begin{Theorem}
Let $\PP$ be a simple thin polyomino such that the reduced Hilbert-Poincar\'e series of $K[\PP]$  is
\[
\HP_{K[\PP]}(t)=\frac{h(t)}{(1-t)^d}.
\]
Then $h(t)$ is the rook polynomial of $\PP$. 
\end{Theorem}
In particular it follows that the Castelnuovo-Mumford regularity of $K[\PP]$ is $r(\PP)$ and the multiplicity of $K[\PP]$  is $r_\PP(1)$.

An open question is to give a complete characterization of the Gorensteinnes of the algebra  $K[\PP]$ when $\PP$ is a simple polyomino. Some partial results in this direction are in \cite{Qu,An,EHQR}.
The other main result of this paper is Theorem \ref{thm:Gore}, in which we classify the simple thin polyominoes $\PP$ having a Gorenstein algebra $K[\PP]$, due to the geometric properties of $\PP$. At the end we present a conjecture and an open question.

\section{Preliminaries}\label{sec:pre}
In this section we recall general definitions and notation on polyominoes and algebraic invariants of commutative algebra (see also \cite{HHO,Vi}).

Let $a = (i, j), b = (k, \ell) \in \NN^2$, with $i	\leq k$ and $j\leq\ell$, the set $[a, b]=\{(r,s) \in \NN^2 : i\leq r \leq k \text{ and } j \leq s \leq \ell\}$ is called an \textit{interval} of $\NN^2$. If $i<k$ and $j < \ell$, $[a,b]$ is called a \textit{proper interval}, and the elements $a,b,c,d$ are called corners of $[a,b]$, where $c=(i,\ell)$ and $d=(k,j)$. In particular, $a,b$ are called \textit{diagonal corners} and $c,d$ \textit{anti-diagonal corners} of $[a,b]$. The corner $a$ (resp. $c$) is also called the left lower (resp. upper) corner of $[a,b]$, and $d$ (resp. $b$) is the right lower (resp. upper) corner of $[a,b]$. A proper interval of the form $C = [a, a + (1, 1)]$ is called a \textit{cell}. Its vertices $V(C)$ are $a, a+(1,0), a+(0,1), a+(1,1)$ and its edges $E(C)$ are 
\[
 \{a,a+(1,0)\}, \{a,a+(0,1)\},\{a+(1,0),a+(1,1)\},\{a+(0,1),a+(1,1)\}.
\]
In the following, we denote by $e(C)$ the left lower corner of a cell $C$.

Let $\PP$ be a finite collection of cells of $\NN^2$, and let $C$ and $D$ be two cells of $\PP$. Then $C$ and $D$ are said to be \textit{connected}, if there is a sequence of cells $C = C_1,\ldots, C_m = D$ of $\PP$ such that $C_i\cap C_{i+1}$ is an edge of $C_i$
for $i = 1,\ldots, m - 1$. In addition, if $C_i \neq C_j$ for all $i \neq j$, then $C_1,\dots, C_m$ is called a \textit{path} (connecting $C$ and $D$). A collection of cells $\PP$ is called a \textit{polyomino} if any two cells of $\PP$ are connected. We denote by $V(\PP)=\cup _{C\in \PP} V(C)$ the vertex set of $\PP$. The number of cells of $\PP$ is called the \textit{rank} of $\PP$, and we denote it by $\rk \PP$.
 A proper interval $[a,b]$ is called an \textit{inner interval} of $\PP$ if all cells of $[a,b]$ belong to $\PP$.
We say that a polyomino $\PP$ is \textit{simple} if for any two cells $C$ and $D$ of $\NN^2$ not belonging to $\PP$, there exists a path $C=C_{1},\dots,C_{m}=D$ such that $C_i \notin \PP$ for any $i=1,\dots,m$. %
An interval $[a,b]$ with $a = (i,j)$ and $b = (k, \ell)$ is called a \textit{horizontal edge interval} of $\PP$ if $j =\ell$ and the sets $\{(r, j), (r+1, j)\}$ for
$r = i, \dots, k-1$ are edges of cells of $\PP$. If a horizontal edge interval of $\PP$ is not
strictly contained in any other horizontal edge interval of $\PP$, then we call it \textit{maximal horizontal edge interval}. Similarly, one defines vertical edge intervals and maximal vertical edge intervals of $\PP$. 

Let $\PP$ be a polyomino and define the polynomial ring $R = K[x_v \ | \ v \in V(\PP)]$ over a field $K$. The binomial $x_a x_b - x_c x_d\in R$ is called an \textit{inner 2-minor} of $\PP$ if $[a,b]$ is an inner interval of $\PP$, where $c,d$ are the anti-diagonal corners of $[a,b]$. We denote by $\MM$ the set of all inner 2-minors of $\PP$. The ideal $I_\PP\subset R$ generated by $\MM$ is called the \textit{polyomino ideal} of $\PP$. We also set $K [\PP] = R/I_\PP$ . 

By combining \cite[Theorem 2.1]{HM} with \cite[Corollary 3.2]{HQS}, one obtains the following
\begin{Lemma}\label{lem:dim}
Let $\PP$ be a simple polyomino. Then $K[\PP]$ is a normal Cohen-Macaulay domain of Krull dimension $|V(\PP)|-\rk \PP$.
\end{Lemma}

Let $R$ be a standard graded ring and $I$ be a homogeneous ideal. The \textit{Hilbert function} $\H_{R/I} : \mathbb{N} \rightarrow \mathbb{N}$ is defined by 
\[
\H_{R/I} (k) := \dim_K (R/I)_k
\]
where $(R/I)_k$ is the $k$-degree component of the gradation of $R/I$, while the \emph{Hilbert-Poincar\'e series} of $R/I$ is
\[
\HP_{R/I} (t) := \sum_{k \in \NN} \H_{R/I}(k) t^k. 
\]
By the Hilbert-Serre theorem, the Hilbert-Poincar\'e series of $R/I$ is a rational function. In particular, by reducing this rational function we get
\[
\HP_{R/I}(t) = \frac{h(t)}{(1-t)^d}.
\]
for some $h(t) \in \mathbb{Z}[t]$, where $d$ is the Krull dimension of $R/I$. The degree of $\HP_{R/I}(t)$ as a rational function, namely $\deg h(t)-d$, is called \emph{a-invariant} of $R/I$, denoted by $a(R/I)$. It is known that whenever $R/I$ is Cohen-Macaulay we have $a(R/I)=\reg R/I-\depth R/I$, that is $\reg R/I=\deg h(t)$.

We recall the following result about Hilbert series
\begin{Proposition}\label{prop:hilbexact}
Let $I$ be a homogeneous ideal of a graded ring $R$, let $f\in R$ be a homogeneous element of degree $d$ and consider the following exact sequence.
\begin{center}
\begin{tikzcd}
    0\arrow{r} & R/(I:f)\arrow{r}{\cdot f} &R/I  \arrow{r} & R/(I,f) \arrow{r} & 0 \\
\end{tikzcd}
\end{center}
Then 
\begin{enumerate}
\item $\HP_{R/I}(t)=\HP_{R/(I,f)}(t)+t^d \HP_{R/(I:f)}(t)$
\item If $f$ is a regular element then
\[
\HP_{R/I}(t)=\frac{1}{1-t^d}\HP_{R/(I,f)}(t).
\]
\end{enumerate}
\end{Proposition}

%\begin{Proposition}\label{prop:betti}
%Let $R=K[x_1,\ldots,x_n]$ and let $M$ be an $R$-module. Then
%\[
%\HP_{M}(t)=\frac{1}{(1-t)^n}\sum\limits_{i=0}^n \sum\limits_{j\in \ZZ} (-1)^i \beta_{ij} t^j
%\]
%\end{Proposition}

We also rephrase the result of Stanley \cite[Theorem 4.4]{St} that is  fundamental for our aim in Section \ref{sec:Gorenstein}.

\begin{Theorem}\label{thm:St}
Let $R=K[x_1,\ldots, x_n]$ be a standard graded polynomial ring, $I$ be a homogeneous ideal of $R$ such that $R/I$ is a Cohen-Macaulay domain, and let 
\[
\HP_{R/I}(t)=\frac{\sum\limits_{i=0}^s h_i t^i}{(1-t)^d}
\]
be the reduced Hilbert series of $R/I$.
Then $R/I$ is Gorenstein if and only if for any $i=0,\ldots,s$ we have $h_i=h_{s-i}$.
\end{Theorem}

\section{Hilbert series of simple thin polyominoes}\label{sec:hilb}
In this section we compute the Hilbert series of simple thin polyominoes in relation with their rook polynomial. We start with the following

\begin{Definition}
Let $C$ and $D$ be two cells of $\mathbb{N}^2$ such that $e(C) \leq e(D)$. We call the set 
\[
[C,D]=\{F \in \NN^2: e(F) \in [e(C),e(D)] \}
\]
\emph{interval of cells}. If $e(C)$ and $e(D)$ lie either on the same vertical edge interval or on the same horizontal edge interval, we call $[C,D]$ a \emph{cell interval}. 	We call $[C,D]$ \emph{inner interval of cells} of $\PP$ if any cell in $[C,D]$ is a cell of $\PP$.
\end{Definition}

\begin{Lemma}\label{lem:maximal}
Let $\PP$ be a simple thin polyomino. Then any maximal inner interval $I$ of cells of $\PP$ is a cell interval, and for any maximal inner  interval $J\neq I$ such that $V(I)\cap V(J)\neq \varnothing$,  $I$ and $J$ have either one cell, one edge or one vertex in common.
\end{Lemma}
\begin{proof}
Since $\PP$ does not contain a square tetromino, then also any maximal inner interval of $\PP$ does not contain a square tetromino, namely it is a cell interval.

Let $I,J$ be two maximal inner intervals of $\PP$ such that $V(I)\cap V(J)\neq \varnothing$. By contradiction, we consider the following two cases: $I$ and $J$ have two or more edges in common, not belonging to the same cell, and $I$ and $J$ have two or more cells in common.
In the first case, without loss of generality $V(I)\cap V(J)=[(i,j),(k,j)]$ with $k>i+1$. Therefore, the cells whose left lower corners are $(i,j-1),(i+1,j-1),(i,j),(i+1,j)$ form a square tetromino, that is a contradiction. \\ 
In the second case, $I\cup J$ is a maximal inner interval  strictly containing  $I$ and $J$, and this is a contradiction.
The assertion follows.
\end{proof}

From now on, we will  briefly call inner intervals the inner intervals of cells of a polyomino $\PP$.
In the following we define the simple polyominoes $\PP'$ and $\PP''$ obtainable from a simple (thin) polyomino $\PP$. The latter are fundamental for the computation of the Hilbert series.

\begin{Definition}[Polyomino $\PP'$]\label{def:P'}
Let $\PP$ be a simple polyomino. We say that a cell $C$ of $\PP$ is a \emph{leaf} if there exists an edge $\{u,v\}$ of $C$  such that $\{u,v\}\cap V(\PP \setminus \{C\})=\varnothing$. We call the vertices $u$ and $v$ \emph{leaf corners} of $C$. We define the polyomino $\PP'$ as the polyomino $\PP\setminus \{C\}$.
\end{Definition}

\begin{Definition}[Polyomino $\PP''$]\label{def:P''}
Let $\PP$ be a simple thin polyomino and let $I$ be a maximal inner interval of $\PP$. We say that $\PP$ is \emph{collapsible} in $I$ if there exists one and only one maximal inner interval $J$ of $\PP$ intersecting $I$ in a cell, and $\PP=\PP_1 \sqcup I \sqcup \PP_2$ where $\PP_1$ and $\PP_2$ are two polyominoes such that $\PP_2$ is either empty or a cell interval. When $\PP_2$ is empty, $I$ is called a \emph{tail}. When $\PP_2$ is a cell interval, $I$ is called an \emph{endcut}. We define the polyomino $\PP''$ as follows. Let $D$ be the cell such that $I\cap J= \{D\}$, and let $\{a,b,a',b'\}$ be the corners of $D$ where $a,b \in V(\PP_1)$ and $a',b' \in V(\PP_2)$ . We define $\PP''$ as the polyomino obtained from $\PP\setminus I$ by the identification of the vertices $a$ and $b$ of $\PP_1$ with the vertices $a'$ and $b'$ of $\PP_2$, respectively, due to the translation of the cell interval $\PP_2$ (see Figure \ref{fig:oper}). 

\end{Definition}
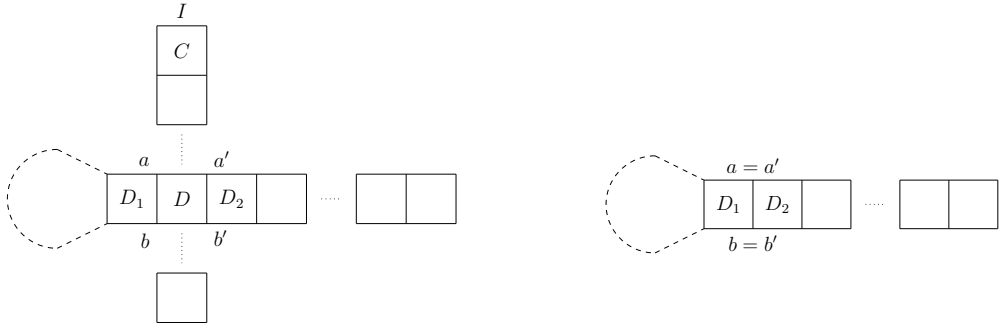
\begin{figure}[H]
\centering
\begin{subfigure}[t]{0.5\textwidth}
\centering
\resizebox{0.8\textwidth}{!}{
\begin{tikzpicture}
\draw (2,2)--(2,1);
\draw (0,1)--(0,2);
\draw (1,0)--(1,-1);
\draw (2,0)--(2,-1);
\draw (1,2)--(1,1);
\draw (3,1)--(3,2);
\draw (4,1)--(4,2);
\draw (5,1)--(5,2);
\draw (6,1)--(6,2);
\draw (7,1)--(7,2);
\draw (1,3)--(1,4);
\draw (2,3)--(2,4);

\draw (1,-1)--(2,-1);
\draw (1,0)--(2,0);
\draw (1,4)--(2,4);
\draw (0,1)--(4,1);
\draw (0,2)--(4,2);
\draw (1,3)--(2,3);
\draw (5,1)--(7,1);
\draw (5,2)--(7,2);

\draw[dashed] (-1,2.5)--(0,2);
\draw[dashed] (-1,0.5)--(0,1);
\draw[dashed] (-1,2.5) arc(90:270:1);

\draw[dotted]   (1.5,2.8)--(1.5,2.2);
\draw[dotted]   (1.5,0.8)--(1.5,0.2);
\draw[dotted]   (4.3,1.5)--(4.7,1.5);

\node at (1.5,5.3) {$I$};
\node at (1,0.65) [anchor=east]{$b$};
\node at (1,2.25) [anchor=east]{$a$};
\node at (2,0.7) [anchor=west]{$b'$};
\node at (2,2.3) [anchor=west]{$a'$};
\node at (1.5,4.5){$C$};
\draw (1,5)--(1,4);
\draw (2,5)--(2,4); 
\draw (1,5)--(2,5); 
\node at (1.5,1.5){$D$};
\node at (0.5,1.5){$D_1$};
\node at (2.5,1.5){$D_2$};

\end{tikzpicture}}\caption{A simple thin polyomino $\PP$ which is collapsible in the endcut $I$}
\end{subfigure}%
\begin{subfigure}[t]{0.5\textwidth}
\centering
\resizebox{0.7\textwidth}{!}{
\begin{tikzpicture}
\draw (0,0)--(3,0);
\draw (4,0)--(6,0);
\draw (0,1)--(3,1);
\draw (4,1)--(6,1);

\draw[dashed] (-1,1.5)--(0,1);
\draw[dashed] (-1,-0.5)--(0,0);
\draw[dashed] (-1,1.5) arc(90:270:1);

\draw (0,0)--(0,1);
\draw (1,0)--(1,1);
\draw (2,0)--(2,1);
\draw (3,0)--(3,1);
\draw (4,0)--(4,1);
\draw (5,0)--(5,1);
\draw (6,0)--(6,1);

\draw[dotted]   (3.3,0.5)--(3.7,0.5);

\node at (1,0) [anchor=north]{$b=b'$};
\node at (1,1) [anchor=south]{$a=a'$};
\node at (0.5,0.5){$D_1$};
\node at (1.5,0.5){$D_2$};
\node at (2,-1.8) [anchor=west]{};
\end{tikzpicture}}\caption{The polyomino $\PP''$ after the collapsing of $\PP$ on  $I$}
\end{subfigure}
\caption{The collapsing operation on a simple thin polyomino $\PP$}\label{fig:oper}
\end{figure}

\begin{Remark} \label{rem:rook}
Let $\PP$ be a simple thin polyomino collapsible in $I$ with leaf $C$. We observe that $r(\PP')\in \{r(\PP),r(\PP)-1\}$ and $r(\PP'')=r(\PP)-1$.
For example, if $\PP$ is the polyomino in Figure \ref{fig:canuzzo} and we consider the leaf $ C_1$, then  $r(\PP ')$ is equal to $r(\PP)-1$. On the other hand, if $\PP$ is the polyomino in Figure \ref{fig:S} and we consider the leaf containing $u$ and $v$, then  $r(\PP ')$ is equal to $r(\PP)$. In both cases, we have $r(\PP'')=r(\PP)-1$. In general, if $C$ belongs to any set of $r(\PP)$ non-attacking rooks, then any set of non-attacking rooks of maximal cardinality in $\PP'$  has $r(\PP)-1$ elements. Otherwise, there exists some set of non-attacking rooks of maximal cardinality in $\PP'$ having $r(\PP)$ elements. Moreover, any set of $r(\PP)$ non-attacking rooks has an element on $I$, that is $r(\PP'')=r(\PP)-1$.
\end{Remark}

We now want to prove that any simple thin polyomino is collapsible in some inner interval $I$. For this aim, we first prove the following

\begin{Lemma}\label{lem:exist}
Let $\PP$ be a simple thin polyomino  that is not a cell interval. Then there exists a maximal inner interval $I$ of $\PP$ for which there exists one and only one maximal inner interval $J$ of $\PP$ intersecting $I$ in a cell.
\end{Lemma}
\begin{proof}
Since $\PP$ is simple and thin, we observe that for any two cells $C$ and $D$ of $\PP$ there is a unique path of cells connecting $C$ and $D$.\\
By contradiction, assume that for any maximal inner interval of $\PP$ there are at least two maximal inner intervals intersecting it in one cell. We show that there exist two different paths connecting two given cells.  For this aim, let $I$ be a maximal inner interval of $\PP$. There exist $I_1$ and $J$ such that $I_1\cap I$ and $I_1\cap J$ are cells of $\PP$. Furthermore, there exists $I_2 \neq I$ intersecting $I_1$ in one cell. By using the same argument, we find a sequence of inner intervals $I_1,I_2,\ldots$ of $\PP$ such that $I_j $ and $ I_{j+1}$ have a cell in common. Since the number of inner intervals of $\PP$ is finite, then there exists $k$ such that $I_k=J$, and hence there are two paths connecting a cell $C$ of $I\setminus I\cap J$ with a cell $D$ of $J\setminus I\cap J$, one passing through $I_1,\ldots, I_{k-1} $ and one passing through the cell $I\cap J$. This is a contradiction and the assertion follows.
\end{proof}

\begin{Proposition}\label{prop:coll}
Let $\PP$ be a simple thin polyomino that is not a cell interval. Then $\PP$ is collapsible in some maximal inner interval $I$.
\end{Proposition}
\begin{proof}
If $\PP$ has a tail, then the assertion follows. Therefore, assume that $\PP$ does not contain tails. \\
By contradiction, assume that $\PP$ has no endcuts. From Lemma \ref{lem:exist}, there exists a maximal inner interval $I_1$ of $\PP$ for which there exists one and only one inner interval $J_1$ of $\PP$ intersecting $I_1$ in one cell. Let $\PP=\PP_1 \sqcup I_1 \sqcup \PP_2$. Since $I_1$ is not an endcut, then $\PP_2$ is a simple thin polyomino that is not a cell interval. Moreover, $\rk \PP_2 < \rk \PP$. Again from Lemma \ref{lem:exist}, there exists an inner interval $I_2$ in $\PP_2 $ for which there exists  one and only one inner  interval $J_2$ of $\PP$ intersecting $I_2$ in one cell. We write $\PP=\PP_3 \sqcup I_2 \sqcup \PP_4 $, with $\PP_1 \subset \PP_3$. We repeat the same argument for the simple thin polyomino $\PP_4$ with $\rk \PP_4 < \rk \PP_2$. By proceeding in this way, since the $\rk \PP$ is finite, at the end we find an inner interval $I_k$ for which $\PP=\PP_{2k-1} \sqcup I_k \sqcup \PP_{2k} $ such that $\rk \PP_{2k}=0$, namely $I_k$ is a tail, that is a contradiction. 
\end{proof}

We observe that the interval $I$ in Lemma \ref{lem:exist} in which $\PP$ is collapsible has one leaf $C$.

\begin{Lemma}\label{lem:colon}
Let $\PP$ be a simple polyomino with a leaf $C$ having leaf corners $u$ and $v$, and let $\PP'$  be as in Definition \ref{def:P'}. Then $((I_\PP,x_u):x_v)=I_{\PP'}+J$ where $J$ is a monomial ideal generated in degree one.
\end{Lemma}
\begin{proof}
Since $C$ is a leaf of $\PP$, then there exists a maximal cell interval $I$ of $\PP$ such that $C \in I$. Let $E=\{u_1, u_2,\ldots,u_r,u\}$ and $F=\{v_1,\ldots,v_r,v\}$ be the edge intervals of length $r+1$ of $I$. 
We observe that the ideal $I_\PP$ is generated by the inner $2$-minors of $\PP'=\PP\setminus \{C\}$ and by the inner $2$-minors of $I$ whose inner intervals contain the cell $C$, namely
\[
I_\PP= I_{\PP'}+(\{x_vx_{u_i}-x_ux_{v_i}\}_{i=1,\ldots,r}).
\]
Then
\[
(I_\PP,x_u)= I_{\PP'}+(\{x_vx_{u_i}\}_{i=1,\ldots,r})+(x_u).
\]
The thesis follows if we prove that $(I_\PP,x_u):x_v \subseteq I_{\PP'}+(x_{u_1},\ldots,x_{u_r},x_u)$, since the other inclusion is trivial. If $f \in (I_\PP,x_u):x_v$, then  $x_vf \in I_{\PP'}+(\{x_vx_{u_i}\}_{i=1,\ldots,r})+(x_u)$, that is 
\[
x_v f = g + x_vg' +x_ug''
\]
where $g \in I_{\PP'}$, $g' \in (x_{u_1},\ldots,x_{u_r})$ and, $g'' \in R$. That is, $x_{v}(f-g')\in I_{\PP'}+(x_u)$ and $f-g'\in (I_{\PP'}+(x_u)):x_v$. Since $\PP'$ is simple, then $I_{\PP'}$ is prime, and since $x_{u}$ is not a variable of $I_{\PP'}$,  then also $I_{\PP'}+(x_{u})$ is prime. Therefore, since $x_v \notin I_{\PP'}+(x_{u})$, then $f-g' \in I_{\PP'}+(x_{u})$ and the assertion follows. 
\end{proof}
\begin{Remark}\label{rem:notbin}
By using the notation of Lemma \ref{lem:colon}, we want to remark that the ideal in the statement has different behaviours, depending on the choice of $u$ and $v$.
Let $\PP$ be the simple thin polyomino in Figure \ref{fig:S}, namely the skew tetromino.

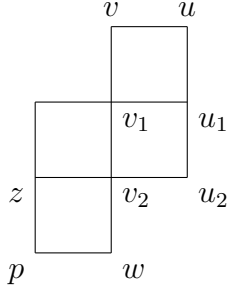
\begin{figure}[H]
\centering 
\begin{tikzpicture}
\draw (1,0)--(2,0);
\draw (1,1)--(3,1);
\draw (1,2)--(3,2);
\draw (2,3)--(3,3);

\draw (1,0)--(1,2);
\draw (2,0)--(2,3);
\draw (3,1)--(3,3);

\node at (2,3) [anchor=south]{$v$};
\node at (3,3) [anchor=south]{$u$};
\node at (2,2) [anchor=north west]{$v_1$};
\node at (3,2) [anchor=north west]{$u_1$};
\node at (2,1) [anchor=north west]{$v_2$};
\node at (3,1) [anchor=north west]{$u_2$};
\node at (2,0) [anchor=north west]{$w$};
\node at (1,1) [anchor=north east]{$z$};
\node at (1,0) [anchor=north east]{$p$};
\end{tikzpicture}
\caption{The skew tetromino}\label{fig:S}
\end{figure}
Since $x_vx_{u_2}-x_ux_{v_2}\in I_{\PP}$, then $x_ux_{v_2} \in (I_{\PP},x_v)$ and $x_{v_2} \in (I_{\PP},x_v):x_u$. Therefore, since $x_px_{v_2}-x_wx_z \in I_\PP$, then $x_wx_z \in (I_{\PP},x_v):x_u$, namely $(I_{\PP},x_v):x_u$ has a monomial generator of degree 2. Nevertheless, the ideal $(I_{\PP},x_u):x_v$ has not monomial generators of degree greater than 1.
\end{Remark}

\begin{Lemma}\label{lem:coronavirus}
Let $\PP$ be a simple thin polyomino, collapsible in $I$ that has $r$ cells, and let $\PP_1,\PP_2,\PP',\PP''$ be as in Definitions \ref{def:P'} and \ref{def:P''}. Let $C$ be a leaf of $I$ with leaf corners $u$ and $v$, and assume that $E=\{u_1, u_2,\ldots,u_r,u\}$ is the edge interval of $I$ such that $E\cap V(\PP_1)= \varnothing$. Then  $R/(I_\PP,x_u,x_v)\cong K[\PP']$ and $R/((I_\PP,x_u):x_v)\cong K[\PP'']\otimes K[y_1,\ldots,y_{r-1}]$.
\end{Lemma}
\begin{proof}
Let $F=\{v_1,\ldots,v_r,v\}$ be the other edge interval of $I$ of length $r+1$.
By the proof of Lemma \ref{lem:colon}, we have
\[
I_\PP= I_{\PP'}+(\{x_vx_{u_i}-x_ux_{v_i}\}_{i= 1,\ldots,r}).
\]
and 
\[
(I_\PP,x_u)= I_{\PP'}+(\{x_vx_{u_i}\}_{i=1,\ldots,r},x_u).
\]
Since $\{u,v\}\cap V(\PP')=\varnothing$, then $(I_\PP,x_u,x_v)=(I_{\PP'},x_u,x_v)$, that is $R/(I_\PP,x_u,x_v)\cong K[\PP']$.

Now let $I''=((I_\PP,x_u):x_v)$. By the proof of Lemma \ref{lem:colon}, it arises $I''= I_{\PP'}+(x_{u_1},\ldots,x_{u_r},x_u)$. Let us consider $J$ and $D$ as in Definition \ref{def:P''}, with  $V(D)=\{u_{k},$ $u_{k+1},$ $v_{k},$ $v_{k+1}\}$. We can split $J$ into the cell intervals $J_1$ and $J_2$,  such that $J_1 \subseteq \PP_1$, $ \PP_2=J_2 $, and the cell $D$. Since the variables $x_{u_1},\ldots, x_{u_r}$, $x_u$ are generators of $I''$, then all of the inner $2$-minors of the interval $I$, and all of the inner $2$-minors of $J$ having corners on $u_{k},u_{k+1}$, are redundant.
Since $\PP_2$ is either empty or a cell interval, then the edge $E$ is a maximal edge interval of $\PP$ (see also Remark \ref{rem:notbin}). We want to prove that $I''$ has not minimal monomial generators of degree greater than 1. By Lemma 	\ref{lem:colon},  assume that there exists a minimal generator $x_wx_z \in I''$, with $w,z \notin \{u_1,\ldots,u_r,u\}=E$. That is there exists $i \in \{1,\ldots, r\}$ and $p \in V(\PP)$ such that $g=x_wx_z-x_{u_i}x_{p}$ is an inner 2-minor of $\PP$. That is one between $w$ and $z$, say $w$, lies on the same edge interval containing the $u_i$'s and $w \notin E$, namely $E \cup \{w\}$ is an edge interval of $\PP$ containing $E$, that is $E$ is not a maximal, contradiction.
 \\
If $\PP_2$ is empty, from Definition \ref{def:P''} we have $\PP''=\PP\setminus I=\PP_1$. Since $E\cap V(\PP_1)=\varnothing$, then $I''= I_{\PP_1}+(x_{u_1},\ldots,x_{u_r},x_u)$ , $V(\PP'') \cap F=\{v_{k},v_{k+1}\}$,  and therefore
\[
R/I''\cong K[\PP'']\otimes K[x_{v_1},\ldots, x_{v_{k-1}},x_{v_{k+2}}\ldots, x_{v_r},x_v]
\]
and the assertion follows.\\
Otherwise, let $\PP''$ be the polyomino arising from the translation of the edge  $\{u_k,u_{k+1}\}$ on the edge $\{v_k,v_{k+1}\}$. We want to prove that $I''=I_{\PP''}+(x_{u_1},\ldots,x_{u_r},x_u)$.

Let $f=f^+-f^- \in I''$ be an irreducible binomial and let
 \[
V(f)=\{v \in V(\PP) \ | \ x_v | f^+ \mbox{ or } x_v | f^-\}.
\]
One of the following is true 
\begin{itemize}
\item[(a)]$V(f) \subseteq V(\PP_1)$ or $V(f)\subseteq V(\PP_2)\setminus \{u_{k},u_{k+1}\}$;
\item[(b)] $|V(f) \cap V(\PP_1)|=|V(f) \cap V(\PP_2)\setminus \{u_{k},u_{k+1}\}|=2$.
\end{itemize} 
In case (a) we have $f \in I_{\PP''}$. \\
In case (b), since $J$ is the unique maximal cell interval having non-empty intersection with both $\PP_1$ and $\PP_2$, we have that $|V(f) \cap V(J_1)|=|V(f) \cap V(J_2)\setminus \{u_{k},u_{k+1}\}|=2$. Since $J_1 \cup J_2$ is a maximal cell interval of $\PP''$, then $f \in I_{\PP''}$. The latter proves $I'' \subseteq I_{\PP''}+(x_{u_1},\ldots,x_{u_r},x_u)$. Similarly the other inclusion follows, due to the fact that an inner interval in $\PP''$ is either an inner interval of $\PP_1$,  of $\PP_2$ (up to the translation defined in Definition \ref{def:P''}), or it is contained in $J_1 \cup J_2$.
Lastly, since $V(\PP'') \cap F=\{v_{k},v_{k+1}\}$, then
\[
R/I''\cong K[\PP'']\otimes K[x_{v_1},\ldots, x_{v_{k-1}},x_{v_{k+2}}\ldots, x_{v_r},x_v]
\]
\end{proof}

\begin{Corollary}\label{cor:HSP}
Let $\PP$ be a simple thin polyomino, collapsible in $I$ that has $r$ cells, with $\PP'$ and $\PP''$ as in Definitions \ref{def:P'} and \ref{def:P''}. Then
 \[
 \HP_{K[\PP]}(t)=\frac{1}{1-t}\Bigg(\HP_{K[\PP']}(t)+\frac{t}{(1-t)^{r-1}}\cdot \HP_{K[\PP'']}(t)\Bigg)
 \]
\end{Corollary} 
\begin{proof}
Let $C$ be a leaf of $I$ and let $u$ and $v$ be the leaf corners of $C$ with $u$ satisfying the hypotheses of Lemma \ref{lem:coronavirus}.
We take the following short exact sequence:
\begin{center}
\begin{tikzcd}
    0\arrow{r} & R/(I_\PP : x_u)  \arrow{r}& R/I_\PP \arrow{r} & R/(I_\PP ,x_u)\arrow{r} & 0 \ \ \ \ \
\end{tikzcd}
\end{center}
Since $\PP$ is simple, then from Lemma \ref{lem:dim} $I_\PP$ is prime, that is $(I_\PP: x_u)=I_\PP$. Therefore, by Proposition \ref{prop:hilbexact}.(2) we have 
\[
\HP_{R/I_\PP}(t)= \frac{1}{1-t} \HP_{R/(I_\PP,x_u)} (t).
\]
We study the Hilbert series of $R/(I_\PP,x_{u})$. By applying Proposition \ref{prop:hilbexact} to the following short exact sequence:
\begin{center}
\begin{tikzcd}
    0\arrow{r} & R/((I_\PP,x_{u}) : x_v)  \arrow{r}& R/(I_\PP,x_{u}) \arrow{r} & R/(I_\PP,x_{u},x_v)\arrow{r} & 0 \ \ \
\end{tikzcd}
\end{center}
we get 
\[
\HP_{K[\PP]}(t)=\frac{1}{1-t}\Bigg(\HP_{R/(I_\PP,x_u,x_v)}(t)+t \cdot \HP_{R/((I_\PP,x_u):x_v)}(t)\Bigg).
\]
Furthermore, by Lemma \ref{lem:coronavirus}, we have 
\begin{enumerate}
 \item $R/(I_\PP,x_u,x_v)\cong K[\PP']$; 
 \item $R/((I_\PP,x_u):x_v)\cong K[\PP'']\otimes K[y_1,\ldots,y_{r-1}]$.
\end{enumerate} 
It is well known that 
\[
\HP_{K[y_{1},\ldots,y_{n}]}(t)=\frac{1}{(1-t)^n}
\]
and 
\[
\HP_{A \otimes B}(t)=\HP_{A}(t)\cdot \HP_{B}(t),
\]
that is 
\[
\HP_{R/((I_\PP,x_u):x_v)}(t)=\frac{1}{(1-t)^{r-1}}\cdot \HP_{K[\PP'']}(t)
\]
and the assertion follows.
\end{proof}

Let $\PP$ be a cell interval with $\rk \PP=r$. The ideal $I_\PP$ can be seen as the determinantal ideal of a $2 \times (r+1)$ matrix. The resolution of the above ideal is well-known (see \cite{BV,Ei}), as well as its Hilbert series. For the sake of completeness, we give the following result 

\begin{Lemma}\label{lem:cellin}
Let $\PP$ be a cell interval with $\rk \PP=r$. Then
\[
\HP_{K[\PP]}(t)=\frac{1+rt}{(1-t)^{r+2}}. 
\]
\end{Lemma}
\begin{proof}
By \cite[Corollary 6.2]{Ei}, $I_\PP$ has linear resolution, and $\beta_{i,i+1}=i\binom{r+1}{i+1}$ for $i=1,\ldots,r$. It is well-known that if $M$ is an $R$-module, then
\[
\HP_{M}(t)=\frac{1}{(1-t)^n}\sum\limits_{i=0}^n \sum\limits_{j\in \ZZ} (-1)^i \beta_{ij} t^j.
\]
That is, the Hilbert series of $K[\PP]$ is
\begin{equation}\label{eq:hilb}
\frac{1+\sum\limits_{i=1}^{r-1}(-1)^i  i \binom{r+1}{i+1}t^{i+1}+(-1)^{r}rt^{r+1}}{(1-t)^{2r+2}}. \tag{$\ast$}
\end{equation}
We study the coefficient $i \binom{r+1}{i+1}$ for $2\leq i \leq r-1$.
\[
i \binom{r+1}{i+1}=(i+1)\binom{r+1}{i+1}-\binom{r+1}{i+1}=
\]
\[
=(r+1)\binom{r}{i}-\binom{r+1}{i+1}=r\binom{r}{i}-\binom{r}{i+1}.
\]
Hence the numerator of Equation \eqref{eq:hilb} becomes
\[
1+\sum\limits_{i=1}^{r-1}(-1)^i \Bigg(r\binom{r}{i}-\binom{r}{i+1} \Bigg)t^{i+1}+(-1)^{r}rt^{r+1}=
\]
\[
=1+\sum_{i=2}^{r}(-1)^i\binom{r}{i} t^{i}+\sum\limits_{i=1}^{r}(-1)^i r\binom{r}{i} t^{i+1}-rt+rt=
\]
\[
(1-t)^r+rt(1-t)^r.
\]
That is 
\[
\HP_{K[\PP]}(t)=\frac{(1+rt)(1-t)^r}{(1-t)^{2r+2}},
\]
and the assertion follows.
\end{proof}

We now state the main theorem (see also Examples \ref{ex:canuzzo} and \ref{ex:stair}). 
\begin{Theorem}\label{thm:rookhilb}
Let $\PP$ be a simple thin polyomino with 
\[
\HP_{K[\PP]}(t)=\frac{h(t)}{(1-t)^d}.
\]
Then $h(t)$ is the rook polynomial of $\PP$. 
\end{Theorem}
\begin{proof}
Let $I_1,\ldots I_s$ be the maximal inner intervals of $\PP$. We proceed by induction on $p=\rk \PP$. \\
If $p=1$, then $\PP$ consists of one cell and by Lemma \ref{lem:cellin}, the statement follows. \\
Let $p>1$ and assume the thesis true for any polyomino with rank less than or equal to $p-1$. 
If $s=1$, then $\PP$ is a cell interval and from by \ref{lem:cellin} we have
\[
\HP_{K[\PP]}(t)=\frac{1+pt}{(1-t)^{p+2}}.
\]
The polynomial $1+pt$ is the rook polynomial of a cell interval having $p$ cells, that is the assertion follows.
If $s>1$, then  $\PP$ is not a cell interval, that is, from Proposition \ref{prop:coll}, $\PP$ is collapsible in some maximal inner interval $I$. Assume that $I$ has $r$ cells. In order to apply Corollary \ref{cor:HSP}, we focus on $\mathrm{HP}_{K[\PP']}(t)$ and $\mathrm{HP}_{K[\PP'']}(t)$. The polyomino  $\PP'$ has $p-1$ cells, while the polyomino $\PP''$ has $p-r$ cells. Hence, from the inductive hypothesis we have 
\[
   \mathrm{HP}_{K[\PP']}(t)=\frac{\sum\limits_{i=0}^{a} r'_i t^i}{(1-t)^{d_1}},
   \]
   where $a=r(\PP)$  with $r'_a \geq 0$ due to Remark \ref{rem:rook},  and $\sum\limits_{i=0}^{a} r'_i t^i$ is the rook polynomial of $\PP'$, and 
   \[
   \mathrm{HP}_{K[\PP'']}(t)=\frac{\sum\limits_{i=0}^{b} r''_i t^i}{(1-t)^{d_2}}.
   \]
   where $b=r(\PP'')=r(\PP)-1$ due to Remark \ref{rem:rook}, and $\sum\limits_{i=0}^{b} r''_i t^i$ is the rook polynomial of $\PP''$.
   From Corollary \ref{cor:HSP} we get 
   \[
   \HP_{K[\PP]}(t)=\frac{1}{1-t}\Bigg(\frac{\sum\limits_{i=0}^{a} r'_i t^i}{(1-t)^{d_1}}+\frac{1}{(1-t)^{r-1}}\frac{\sum\limits_{i=0}^{b}r''_i t^{i+1}}{(1-t)^{d_2}}\Bigg)=\frac{\sum\limits_{i=0}^{a} r'_i t^i}{(1-t)^{d_1+1}}+\frac{\sum\limits_{i=0}^{b}r''_i t^{i+1}}{(1-t)^{d_2+r}}
   \]
We first show that $d_1+1=d_2+r=n-p$, where $n=|V(\PP)|$. Since $\PP'$ is the polyomino having $n-2$ vertices and $p-1$ cells, then from Lemma \ref{lem:dim} we have $(n-2)-(p-1)=n-p-1$. Moreover, since $I$ is on the $2r+2$ vertices $\{x_1,\ldots, x_r,x,y_1,\ldots,y_r,y\}$ but $y_k,y_{k+1}$ for some $k$ are corners of one cell of $\PP\setminus I$, then $\PP''$ is the polyomino having $n-2r$ vertices and $p-r$ cells, hence from Lemma \ref{lem:dim} $d_2+r-1=(n-2r)-(p-r)+r-1=n-p-1$. 
That is 
\[
 \HP_{K[\PP]}(t)=\frac{1+\sum\limits_{i=1}^{r(\PP)}(r'_{i}+r''_{i-1})t^i}{(1-t)^d}
\]
For $1\leq i\leq r(\PP)$, $r_i=r'_{i}+r''_{i-1}$. In fact, $r_i$ is the number of ways of placing $i$ non-attacking rooks on all of the cells of $\PP$,  whereas $r'_i$ is the number of ways of placing $i$ non-attacking rooks on the simple thin polyomino $\PP'$, namely the number of ways of placing $i$ non-attacking rooks on the cells $D\neq C$ of $\PP$, and  $r''_{i-1}$  is the number of ways of placing $i-1$ non-attacking rooks on the simple thin polyomino  $\PP''$, namely the number of ways of placing $i-1$ non-attacking rooks on the cells $D$ of $\PP$ such that $D \notin I$, given that the $i$-th rook is placed on the cell $C$, hence the thesis follows.
\end{proof}
We immediately deduce the following
\begin{Corollary}
Let $\PP$ be a simple thin polyomino. Then the Castelnuovo-Mumford regularity is $r(\PP)$ and the multiplicity of $K[\PP]$ is $r_\PP(1)$.
\end{Corollary}
\begin{Remark}\label{rem:nonthin}
In general the equality  $h(t)=r_{\PP}(t)$ does not hold for any simple polyomino $\PP$. Let $\PP$ be the square tetromino. Then, by using Macaulay2 we find that \[
h(t)=1+4t+t^2 \mbox{ and } r_\PP(t)=1+4t+2t^2.
\]
Even though the two polynomials are different, they have the same degree, that is $\reg K[\PP]=r(\PP)$ also in this case.
\end{Remark}

\section{Gorenstein simple thin polyominoes}\label{sec:Gorenstein}
In this section we characterize the Gorenstein simple thin polyominoes.
We start with a fundamental definition for our goal.
\begin{Definition}
Let $\PP$ be a simple thin polyomino.
A cell $C$ of $\PP$ is \emph{single} if there exists a unique maximal inner interval of $\PP$ containing $C$. If any maximal inner interval of $\PP$ has exactly one single cell, we say that $\PP$ has the \emph{S-property}.
\end{Definition}

Let $\CC$ be the set of the single cells of a simple thin polyomino. We set $\DD$ as the collection of cells $\PP\setminus \CC$. In particular since $\PP$ is thin,  then any cell of $\DD$ belongs exactly to two maximal inner intervals of $\PP$.

\begin{Theorem}\label{thm:Gore}
Let $\PP$ be a simple thin polyomino, $I_1$, $\ldots,$ $I_s$ be its maximal inner intervals,
and let $r_{\PP}(t)=\sum_{k=0}^s r_k t^k$ be its rook polynomial. Then the following conditions are equivalent:
\begin{itemize}
\item[\mbox{(a)}] $K[\PP]$ is Gorenstein;
\item[\mbox{(b)}]$\forall i=0,\ldots,s$ we have  $r_i=r_{s-i}$;
\item[\mbox{(c)}] $\PP$ satisfies the $S$-property.
\end{itemize}
\end{Theorem}
\begin{proof} 
(a)$\Leftrightarrow$(b): By combining Theorem \ref{thm:St} and Theorem \ref{thm:rookhilb}, for a simple thin polyomino $\PP$, the Cohen-Macaulay domain $K[\PP]=R/I_{\PP}$ is Gorenstein if and only if $\forall i=0,\ldots,s$ we have  $r_i=r_{s-i}$, and the assertion follows.\\
(c)$\Rightarrow$(b): Since $\PP$ satisfies the $S$-property, then any maximal inner interval $I$ of $\PP$ contains a unique single cell $C$. Therefore, let $\CC=\{C_1, \ldots C_s\}$ be the set of the single cells of $\PP$, and let $I_1,\ldots, I_s$ be the maximal inner intervals of $\PP$ such that $C_i \in I_i$. We set $\DD=\PP\setminus \CC$. As we have observed above, any cell of $\DD$ is the intersection of two maximal inner intervals of $\PP$, and we denote by $D_{jk}$ the cell of $\DD$ in the intersection of $I_j$ and $I_k$.

Let ${\bf i}$ be a subset of $[s]$ of cardinality $l$, and let ${\bf jk}=\{ \{j_1,k_1\},\ldots, \{j_m,k_m\}\}$ with $j_t,k_t\in [s]$ for $1\leq t\leq m$. We denote by $\CC_{\bf i}=\{C_i \in \CC :  \ i\in \bf i\}$ and by $\DD_{\bf jk}=\{D_{jk}\in \DD: \{j,k\}\in {\bf jk}\}$. 
%by $

Let ${\bf j}=\{j_1,\ldots j_{m}\}$ and ${\bf k}=\{k_1,\ldots k_{m}\}$ be such that $\bf j\cap \bf k=\varnothing$ and let $\bf i$ be such that $\bf i \cap (j\sqcup k)=\varnothing$ then 
\begin{equation}\label{eq:non-attackingSet}
 \CC_{\bf i}\cup \DD_{\bf jk}
\end{equation}
induces a set of $d=l+m$ non-attacking rooks, and any set of non-attacking rooks of cardinality $d$ can be written in the form \eqref{eq:non-attackingSet}, and this configuration is unique because a set ${\bf jk}$ identifies a unique subset of $\DD$ and thanks to the $S$-property a set $\bf i\subset [s]$ identifies a unique subset of $\CC$.
Our goal is to prove that for any configuration \eqref{eq:non-attackingSet} of $d$  non-attacking rooks  there exists a unique configuration of the form \ref{eq:non-attackingSet} of $s-d$ non-attacking rooks. Let $\overline{\CC}_{{\bf i} \cup {\bf i} \cup {\bf k} }=\CC\setminus (\CC_{\bf i} \cup \CC_{\bf j} \cup \CC_{\bf k} )$, and since $\bf i \cap (j\cup k)=\varnothing$, then $|\overline{\CC}_{\bf i \cup  j \cup k}|=s-(l+2m)$. From the configuration of cardinality $d$ in \eqref{eq:non-attackingSet}, we retrieve the following configuration of cardinality $s-d$,
\begin{equation}\label{eq:non-attackingSetDual}
\overline{\CC}_{\bf i \cup  j \cup k} \cup \DD_{\bf jk}.
\end{equation}
In fact, $s-(l+2m)+m=s-d$ and the configuration \eqref{eq:non-attackingSetDual} satisfies the properties of configuration \eqref{eq:non-attackingSet}, and the configuration \eqref{eq:non-attackingSetDual} is uniquely determined by \eqref{eq:non-attackingSet} because $\DD_{\bf jk}$ is fixed, and once we set $\CC_{\bf i}$ and $\bf{j\cup k}$, the complement set $\overline{\CC}_{\bf i \cup  j \cup k}$ is unique.

(b)$\Rightarrow$(c): By contraposition, assume that $\PP$ does not satisfy the $S$-property, that is there exists an inner interval $I$ of $\PP$ having $q$ single cells with $q\neq 1$. We want to prove that either $r_{s}>r_{0}=1$ or $r_{s-1}>r_{1}=\rk \PP$. 

Let $q>1$, and let $C,C'$ be two single cells of $I$. Any set $\CC$ of $s$ non-attacking rooks contains a single cell $C''$ of $I$ such that either $C''\neq C$ or $C'' \neq C'$. In both cases the sets $\CC\setminus \{C''\} \cup C$ and $\CC\setminus \{C''\} \cup C'$ are two distinct sets of $s$ non-attacking rooks, that is $r_s>1$, and it is a contradiction. \\
Hence, from now on we assume that in $\PP$ do not exist maximal inner intervals with two or more single cells. That is, any maximal inner interval of $\PP$ has either $0$ or $1$ single cells and in particular we assume $q=0$. Let $\CC$ be a set of $s$ non-attacking rooks of $\PP$. In this case one of the following is true:
\begin{enumerate}
\item any inner interval $J$ intersecting $I$ in a cell $D$ contains a cell $C\neq D$ such that $C\in \CC$, in particular $I \cap \CC= \varnothing$;
\item there exists an inner interval $J$ intersecting $I$ in a cell $D\in \CC$.
\end{enumerate}
In case (1), $(\CC \setminus \{C\})\cup \{D\}$ is a set of $s$ non-attacking rooks different from $\CC$,  that is $r_s>1$, and it is a contradiction.

In case (2), we want to show $r_{s-1}>r_1$. Let $E$ be a cell of $\PP$. If $E \in \CC$, then $\CC\setminus \{E\}$ is a set of $s-1$ non-attacking rooks. If $E \notin \CC$, then $E$ is not single, that is $E$ is the intersection of two cell intervals $I_1$ and $I_2$. From the maximality of $\CC$, there exist two cells $F\in I_1$ and  $G\in I_2$ with $F,G \in \CC$, and $\CC\setminus \{F,G\}\cup \{E\}$ is a set of $s-1$ non-attacking rooks. Hence $r_{s-1}\geq r_1$.
	
The hypothesis (2) implies that there exist some cells  $A,B,C_1,C_2$ of $\PP$ such that the polyomino  $\mathcal{Q}$ in Figure \ref{fig:stairs} is a subpolyomino of $\PP$ (up to rotations and reflections). In fact, without loss of generality assume that $A$ is a cell of $I$ and $B$ is a cell of $J$. Since $I$ has no single cells there exists an inner interval $J'$ intersecting $I$ in $A$. Moreover, if the cell $B$ is single, then $B \in \CC$ and this contradicts (2). Hence there exists an inner interval $J''$ intersecting $J$ in $B$.

\begin{figure}[H]
\centering
\resizebox{0.3\textwidth}{!}{
\begin{tikzpicture}
\draw (0,3)--(2,3);
\draw (0,2)--(3,2);
\draw (1,1)--(3,1);
\draw (2,0)--(3,0);

\draw (0,3)--(0,2);
\draw (1,3)--(1,1);
\draw (2,3)--(2,0);
\draw (3,2)--(3,0);

\node  at (1.5,2.5) {$A$}; 
\node  at (1.5,1.5) {$D$}; 
\node  at (2.5,1.5) {$B$}; 
\node  at (0.5,2.5) {$C_1$}; 
\node  at (2.5,0.5) {$C_2$}; 
\end{tikzpicture}}\caption{A simple thin polyomino $\mathcal{Q}$ that does not satisfy the $S$-property}\label{fig:stairs}

\end{figure}
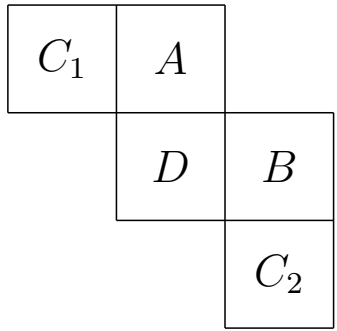

Let $F$ and $G$ be the cells of $\CC$ that belong respectively to $J'$ and $J''$.
We consider the following sets of $s-1$ non-attacking rooks:
\[
 \CC\setminus \{F,D\} \cup \{A\},\CC\setminus \{G,D\} \cup \{B\},\ \CC\setminus \{F,G,D\} \cup \{A,B\}.
\]
The first two were mentioned in the discussion above, while the third one increases the number $r_{s-1}$. Hence $r_{s-1}>r_1$, that is a contradiction.

\end{proof}
\begin{Example}\label{ex:canuzzo}
Let $\PP$ be the polyomino in Figure \ref{fig:canuzzo}. 
\begin{figure}[H]
\centering
\resizebox{0.3\textwidth}{!}{
\begin{tikzpicture}
\draw (0,2)--(0,3);
\draw (1,0)--(1,3);
\draw (2,0)--(2,3);
\draw (3,0)--(3,2);
\draw (4,0)--(4,2);

\draw (1,0)--(2,0);
\draw (3,0)--(4,0);
\draw (1,1)--(4,1);
\draw (0,2)--(4,2);
\draw (0,3)--(2,3);

\node  at (0.5,2.5) {$C_1$}; 
\node  at (1.5,2.5) {$D_{12}$}; 
\node  at (1.5,1.5) {$D_{23}$}; 
\node  at (1.5,0.5) {$C_2$}; 
\node  at (2.5,1.5) {$C_3$}; 
\node  at (3.5,1.5) {$D_{34}$}; 
\node  at (3.5,0.5) {$C_4$}; 
\end{tikzpicture}}\caption{A simple thin polyomino satisfying the $S$-property}\label{fig:canuzzo}
\end{figure}
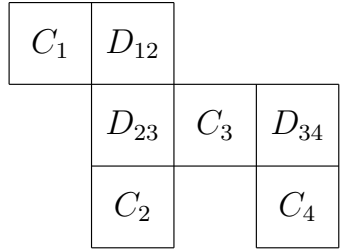
We see that $\PP$ has $4$ maximal inner intervals and a single cell for any of these ones, that is $\PP$ satisfies the $S$-property. We want to compute the Hilbert series of $K[\PP]$. It is easy to see that $r(\PP)=4$. According to Theorem \ref{thm:rookhilb}, the Hilbert series of $K[\PP]$ is 
\[
\HP_{K[\PP]}(t)=\frac{\sum\limits_{i=0}^4 r_it^i}{(1-t)^d}
\]
where $d=|V(\PP)|-\rk \PP=16-7=9$. We compute $r_{i}$, namely the number of sets of $i$ non-attacking rooks for $i=0,\ldots,4$.
\begin{enumerate}
\item[($i=0$)] $\varnothing$;
\item[($i=1$)] $\{C_1\},\{C_2\},\{C_3\},\{C_4\},\{D_{12}\},\{D_{23}\},\{D_{34}\}$;
\item[($i=2$)] $\{C_1,D_{23}\},\{C_1,C_2\},\{C_1,C_3\},\{C_1,D_{34}\},\{C_1,C_4\},\{D_{12},C_3\},\{D_{12},D_{34}\},  $ \\ $\{D_{12},C_4\}, \{C_2,C_3\}, \{C_2,D_{34}\},\{C_2,C_4\}, \{D_{23},C_4\},\{C_3,C_4\} $;
\item[($i=3$)] $\{C_1,C_2,C_3\},\{C_1,C_2,C_4\},\{C_1,C_3,C_4\},\{C_2,C_3,C_4\},\{C_1,C_2,D_{34}\},\{C_1,D_{23},C_4\}$\\ $\{D_{12},C_{3},C_4\}$; 
\item[($i=4$)] $\{C_1,C_2,C_3,C_4\}$.
\end{enumerate}
It follows
\[
r_{0}=1, \ r_{1}=7, \ r_{2}=13 , \ r_{3}=7 , \ r_{4}=1,
\]
that is 
\[
\HP_{K[\PP]}(t)=\frac{1+7t+13t^2+7t^3+t^4}{(1-t)^9}
\]
and according to Theorem \ref{thm:St}, $K[\PP]$ is Gorenstein.
\end{Example}
\begin{Example}\label{ex:stair}
In the notation of Theorem \ref{thm:Gore}, we highlight that the condition $r_s=1$ is not sufficient to guarantee that the polynomial has symmetric coefficients. In fact, let us consider the polyomino $\mathcal{Q}$ in Figure \ref{fig:stairs}. The rook number of $\mathcal{Q}$ is $3$ and the rook polynomial of $\mathcal{Q}$ is 
\[
1+5t+6t^2+t^3,
\]
in fact, the sets of $i$ non-attacking rooks are
\begin{enumerate}
\item[($i=0$)] $\varnothing$;
\item[($i=1$)] $\{A\},\{B\},\{C_1\},\{D\},\{C_2\}$;
\item[($i=2$)] $\{C_1,D\},\{C_1,C_2\},\{D,C_2\},\{B,C_1\},\{A,C_2\},\{A,B\}$;
\item[($i=3$)] $\{C_1,D,C_2\}$;
\end{enumerate}
As already noted in the proof of Theorem \ref{thm:Gore} the fact that $r_2>r_1$ depends on the set $\{A,B\}$. 
\end{Example}

To conclude the paper, we want to remark that among the thin polyominoes that are not simple, namely multiply-connected, there are some non-prime ones, so that we can not directly retrieve the Cohen-Macaulayness of $K[\PP]$. Nevertheless, due to Theorem \ref{thm:rookhilb} and Remark \ref{rem:nonthin}, we conjecture the following 

\begin{Conjecture}
Let $\PP$ be a polyomino. Then $\PP$ is thin if and only if $r_{\PP}(t)=h(t)$.
\end{Conjecture}

Moreover, due to Theorem \ref{thm:rookhilb} and \cite[Theorem 2.3]{EHQR}, we ask the following
\begin{Question}
Let $\PP$ be a polyomino. Then $\reg K[\PP]=r(\PP)$?
\end{Question}

\end{document}